\documentclass[english,11pt]{article}

\usepackage{amssymb,amsmath,amsfonts,amsthm,rotating,setspace,graphicx,float}
\usepackage{amssymb,amsmath,amsthm,rotating,setspace}
\usepackage{enumerate}

\usepackage{latexsym}

\usepackage{epsfig}

\usepackage{color}

\newtheorem{theorem}{Theorem}[section]
\newtheorem{lemma}[theorem]{Lemma}

\theoremstyle{definition}

\theoremstyle{definition}
\newtheorem{remark}[theorem]{Remark}

\theoremstyle{proposition}
\newtheorem{proposition}[theorem]{Proposition}

\numberwithin{equation}{section}

\def\eq#1{(\ref{#1})}

\def\R{\mathbb{R}}

\begin{document}

\title{Regularity of Weak Solutions for Singular Elliptic Problems Driven by $m$-Laplace Operator}

\author{{\large Gurpreet Singh}\\School of Mathematics and Statistics\\
 University College Dublin,\\ Belfield, Dublin 4, Ireland\\E-mail: {\tt gurpreet.singh@ucdconnect.ie}}







\maketitle

\begin{abstract}
We obtain optimal regularity in the Sobolev space $W_0^{1,\tau}(\Omega)$ for the unique solution of 
$$
-\Delta_m u=K(x)u^{-p} \mbox{ in } \Omega, \quad u=0\mbox{ on }\partial \Omega.
$$ 
Here $\Omega\subset{\mathbb R}^N$ is a smooth and bounded domain, $m>1$, $p\geq 0$ and $K\in C(\Omega)$ is a positive function that behaves like ${\rm dist}(x,\partial\Omega)^{-q}$ for some $q\geq 0$ with $p+q< 2-\frac{1-p}{m}$. 

We obtain that the unique weak solution to the above problem belongs to   $W_0^{1,\tau}(\Omega)$ for 
$$
m\leq \tau<\frac{m+p-1}{p+q-1}\quad \mbox{ if }p+q>1,
$$
and
$$
m\leq \tau<\infty\quad \mbox{ if }p+q=1.
$$
The above range of $\tau$ is optimal.
\end{abstract}


\section{Introduction and the main result}

In this note we are interested in the study of regularity of the weak solution to 
\begin{equation}\label{eq8}
\left\{
\begin{aligned}
-\Delta_{m} u&= K(x)u^{-p} &&\quad\mbox{ in }\Omega,\\
u&= 0 &&\quad\mbox{ on }\partial{\Omega}.
\end{aligned}
\right.
\end{equation}
Here $\Omega\subset{\mathbb R}^N$ is a smooth and bounded domain, $m>1$, $p\geq 0$ and $K\in C(\Omega)$ is a positive function that behaves like $\delta(x)^{-q}$ for some $q\geq 0$, with $p+q< 2-\frac{1-p}{m}$ where $\delta (x):={\rm dist}(x,\partial\Omega)$. The existence of a weak solution $u\in W^{1,m}_0(\Omega)\cap C^{1,\alpha}(\overline\Omega)$ (if $p+q<  1$) or $u\in W^{1,m}_0(\Omega)\cap C^{0,\beta}(\overline\Omega)$ (if $p+q\geq 1$) has been obtained in \cite{BGH2015, GMS2012, GST2015} for some $\alpha$, $\beta \in (0, 1)$. We are here interested in the optimal $W^{1,\tau}_0(\Omega)$ regularity of the unique solution.

\medskip

In this note, for any two functions $f$ and $g$ defined on $\Omega$ we shall write $f\sim g$ to denote that
$$
c_{1}\leq \frac{f}{g}\leq c_{2} \quad\mbox{ in }\Omega,
$$
for some positive constants $c_{1}$ and $c_{2}$. 

\medskip

Our main result in this note is as follows:

\begin{theorem}\label{tm1}
Let $p\geq 0$ and $K: \Omega \rightarrow (0, \infty)$ be a continuous function such that $K(x)\sim \delta(x)^{-q}$, $p+q< 2-\frac{1-p}{m}$. Then, the problem 
\begin{equation}\label{eq8a}
\left\{
\begin{aligned}
-\Delta_{m} u&= K(x)u^{-p} &&\quad\mbox{ in }\Omega,\\
u&= 0 &&\quad\mbox{ on }\partial{\Omega},
\end{aligned}
\right.
\end{equation}
has a unique solution $u\in W_{0}^{1, m}(\Omega)$ and:
\begin{itemize}
\item[{\rm (i)}] If $p+q<1$ then $u\in C^{1, \alpha}(\overline{\Omega})$ for some $\alpha \in (0, 1)$ and $u(x)\sim \delta(x)$.
\item[{\rm (ii)}] If $p+q=1$ then $u\in W_{0}^{1, \tau}(\Omega)\cap C^{0, \beta}(\overline{\Omega})$ for some $\beta \in (0, 1)$ and for all $m\leq \tau< \infty$. Also $u\sim \delta(x) \log^{\frac{1}{m+p-1}}\Big(\frac{1}{\delta(x)}\Big)$. 
\item[{\rm (iii)}] If $p+q>1$ then $u\in W_{0}^{1, \tau}(\Omega)\cap C^{0, \beta}(\overline{\Omega})$ for some $\beta \in (0, 1)$ and for all $m\leq \tau< \frac{m+p-1}{p+q-1}$. Also $u\sim \delta(x)^{\frac{m-q}{m+p-1}}$. 
\end{itemize}
\end{theorem}

\section{Auxiliary results}
The key result in proving Theorem \ref{tm1} is the following:
\begin{proposition}\label{p1}
Let $u\in W_{0}^{1, m}(\Omega)\cap C(\overline{\Omega})$, $m>1$ satisfy 
\begin{equation}\label{eq1}
\left\{
\begin{aligned}
-\Delta_{m}u&= \theta(x) &&\quad\mbox{ in }\Omega,\\
u&= 0 &&\quad\mbox{ on }\partial{\Omega},
\end{aligned}
\right.
\end{equation}
where $\theta: \Omega \rightarrow (0, \infty)$ is a continuous function.
\begin{itemize}
\item[{\rm (i)}] If $\theta(x)\sim \delta(x)^{-a}$ for $a\in \Big(1, 2-\frac{1}{m}\Big)$ then $u\in W_{0}^{1, p}(\Omega)$ for all $p\in \Big[m, \frac{m-1}{a-1}\Big)$.
\item[{\rm (ii)}] If $\theta(x)\sim \delta(x)^{-1}log^{-a}\Big(\frac{1}{\delta(x)}\Big)$ for $a\in (0, 1)$ then $u\in W_{0}^{1, p}(\Omega)$ for all $p\in [m, \infty)$.
\end{itemize}
\end{proposition}

\medskip

We shall use the following results:
\begin{lemma}\label{l1} {\rm (see \cite[Theorem 2]{I1983}). }
Assume $p\geq m> 1$, $u\in W_{0}^{1, m}(\Omega)$ and $\Phi \in L^{\frac{p}{m-1}}(\Omega; \R^{N})$ satisfy 
$$
\Delta_{m}u= div(\Phi) \quad\mbox{ in }\Omega.
$$
Then $\nabla u\in L^{p}(\Omega; \R^{N})$ and there exist $c= c(m, p, N)$ such that
$$
||\nabla u||_{L^{p}(\Omega)}^{m-1}\leq c||\Phi||_{L^{\frac{p}{m-1}}(\Omega)}.
$$
\end{lemma}
\begin{lemma}\label{l2}{\rm (see \cite[Lemma 2]{CRT1977} or \cite[Lemma 4.4]{BGH2015}). }
There exist $c>0$ such that if $B_{2r}(x_{0})\subset \Omega$, $0< r\leq 1$ and $v\in W_{0}^{2, p}(\Omega)$, for some $p> N$, then
\begin{equation}
||\nabla v||_{L^{\infty}(B_{r}(x_{0}))}\leq c\Big[r||\Delta v||_{L^{\infty}(B_{2r}(x_{0}))}+ \frac{1}{r}||v||_{L^{\infty}(B_{2r}(x_{0}))}\Big].
\end{equation}
\end{lemma}
\begin{lemma}\label{l3} {\rm (see \cite{LM1991}). }
Let $\Omega \subset \R^{N}$ be a bounded and smooth domain. Then
$$
\int_{\Omega}\delta(x)^{-a}dx< \infty \quad\mbox{ if and only if } a< 1.
$$
\end{lemma}
\begin{proof}[Proof of Proposition 2.1]
Let $w\in C^{2}(\Omega)\cap C(\overline{\Omega})$ satisfy 
\begin{equation}\label{eq2}
\left\{
\begin{aligned}
-\Delta w&= \theta(x) &&\quad\mbox{ in }\Omega,\\
w&=0 &&\quad\mbox{ on }\partial{\Omega}.
\end{aligned}
\right.
\end{equation}
Denote by $\phi>0$ the first eigenfunction of $-\Delta$ in $\Omega$.
\begin{itemize}
\item[{\rm (i)}] Assume $\theta(x)\sim \delta(x)^{-a}$ for some $a\in \Big(1, 2-\frac{1}{m}\Big)$. Then
$$
\underline{w}(x)= \frac{1}{c}\phi(x)^{2-a} \quad\mbox{ and } \quad \overline{w}(x)= c\phi(x)^{2-a},
$$ 
are respectively sub and supersolutions of \eq{eq2} provided $c>1$ is large enough, so 
\begin{equation}\label{eq3}
w(x)\sim \delta(x)^{2-a}.
\end{equation}
We claim that
\begin{equation}\label{eq4}
|\nabla w(x)| \leq c\delta(x)^{1-a} \quad\mbox{ in }\Omega,
\end{equation}
for some $c>0$. In order to prove this, let $x\in \Omega$ be a fixed point and $r=\frac{\delta(x)}{3}$. Then 
$$
B_{2r}(x)\subset \Omega_{0}= \{ z\in \Omega: \frac{\delta(x)}{3}\leq \delta(z) \leq \frac{5}{3}\delta(x) \} \subset \Omega
$$
and by Lemma \ref{l2} we have:
\begin{equation}\label{eq5}
\begin{aligned}
|\nabla w(x)|\leq c\Big[r||\Delta w||_{L^{\infty}(\Omega_{0})}+ \frac{1}{r}||w||_{L^{\infty}(\Omega_{0})}\Big]\leq c\delta(x)^{1-a}
\end{aligned}
\end{equation}
which proves \eq{eq4}. Using the estimate in \eq{eq4} we deduce that $|\nabla w|\in L^{\frac{p}{m-1}}(\Omega)$ whenever $\delta(x)^{1-a}\in L^{\frac{p}{m-1}}(\Omega)$ and by Lemma \ref{l3} this is equivalent to $p< \frac{m-1}{a-1}$. Using Lemma \ref{l1} with $\Phi= \nabla w$, we conclude the proof.
\item[{\rm (ii)}] Assume now $\theta(x)\sim \delta(x)^{-1}\log^{-a}\Big(\frac{1}{\delta(x)}\Big)$ for some $a\in (0, 1)$. Then 
\begin{equation*}
\underline{w}(x)= \frac{1}{c}\phi(x)\log^{1-a}\Big(\frac{A}{\phi(x)}\Big) \quad\mbox{ and } \quad \overline{w}(x)= c\phi(x) \log^{1-a}\Big(\frac{A}{\phi(x)}\Big),
\end{equation*}
are respectively sub and supersolutions of \eq{eq2}, where $A>1$ is large. So,
\begin{equation}\label{eq6}
w(x)\sim \delta(x)\log^{1-a}\Big(\frac{A}{\delta(x)}\Big). 
\end{equation}
Using \eq{eq6} and the similar approach as in part(i), we deduce that 
$$
|\nabla w(x)|\leq c\log^{1-a}\Big(\frac{A}{\delta(x)}\Big) \quad\mbox{ in }\Omega,
$$
where $A> 1+diam(\Omega)$. In particular $|\nabla w|\in L^{p}(\Omega)$ for all $p>1$ which, by Lemma \ref{l1} with $\Phi= \nabla w$, yields $u\in W_{0}^{1, p}(\Omega)$ for all $p\in [m, \infty)$. This finishes the proof of our result.
\end{itemize}
\end{proof}

\begin{remark}
The regularity in Proposition \ref{p1} is optimal. In order to see this, let $(\varphi, \lambda)$ denote the first eigenfunction and eigenvalue of $-\Delta_{m}$ in $\Omega$, that is
\begin{equation}\label{eq7}
\left\{
\begin{aligned}
-\Delta_{m} \varphi&= \lambda |\nabla \varphi|^{m-2}\varphi &&\quad\mbox{ in }\Omega,\\
\varphi&= 0 &&\quad\mbox{ on }\partial{\Omega}.
\end{aligned}
\right.
\end{equation}
It is known that $\lambda>0$, $\varphi \in C^{2}(\Omega)\cap C^{1}(\overline{\Omega})$ and $\varphi$ has constant sign in $\Omega$.  Thus, by normalizing $\varphi$ we may assume $\varphi> 0$ in $\Omega$ and $||\varphi||_{\infty}= 1$. To show that $W_{0}^{1, p}(\Omega)$ regularity in Proposition \ref{p1}(i) is optimal, let $\theta(x)= -\Delta_{m}(\varphi^{2-a})$. Some straightforward calculations yield $\theta(x)\sim \varphi(x)^{-a}\sim \delta(x)^{-a}$. Thus, $w= \varphi^{2-a}$ is a solution of \eq{eq1} with $\theta(x)$ given by $-\Delta_{m}(\varphi^{2-a})$. Clearly $w\in W_{0}^{1, p}(\Omega)$  for all $m\leq p< \frac{m-1}{a-1}$, but by Lemma \ref{l3} one has $w\not\in W_{0}^{1, \frac{m-1}{a-1}}(\Omega)$.

\medskip

Similarly, to show that regularity $w\in W_{0}^{1, p}(\Omega)$, $m\leq p< \infty$ is optimal we take $\theta(x)= -\Delta_{m}\Big(\varphi \log^{1-a}\Big(\frac{A}{\varphi}\Big)\Big)$ where $A>1$ is a large constant. 
\end{remark}

\section{Proof of Theorem \ref{tm1}}
The existence of a solution $u\in W_{0}^{1, m}(\Omega)$ follows from {\rm \cite[Theorem 3.2]{GST2015}}.
\begin{itemize}
\item[{\rm (i)}] If $p+q<1$, then by {\rm \cite[Theorem 2.1]{GST2015}}, we have $u\in C^{1, \alpha}(\overline{\Omega})$, for some $\alpha \in (0, 1)$.
\item[{\rm (ii)}] If $p+q=1$, then by {\rm \cite[Theorem 2.1]{GST2015}}, we have $u\in W_{0}^{1, m}(\Omega)\cap C^{0, \beta}(\overline{\Omega})$ for some $\beta \in (0, 1)$. Also the behaviour $u\sim \delta(x)\log^{\frac{1}{m+p-1}}\Big(\frac{1}{\delta(x)}\Big)$ follows in the same way as in {\rm \cite[Lemma 3.3]{GMS2012}} by noting that
\begin{equation*}
\underline{u}(x)= \frac{1}{c}\varphi(x)\log^{1-a}\Big(\frac{A}{\varphi(x)}\Big) \quad\mbox{ and } \quad \overline{u}(x)= c\varphi(x) \log^{1-a}\Big(\frac{A}{\varphi(x)}\Big),
\end{equation*}
are respectively sub and supersolutions for some large $c>1$. Using the asymptotic behaviour of $u$ we deduce that 
$$
\theta(x)= K(x)u^{-p}(x)\sim \delta(x)^{-1}\log^{-\frac{p}{m+p-1}}\Big(\frac{1}{\delta(x)}\Big).
$$
By Proposition \ref{p1}(ii) it follows that $u\in W_{0}^{1, \tau}(\Omega)$ for all $\tau \in [m, \infty)$.
\item[{\rm (iii)}] If $p+q>1$, then by {\rm \cite[Theorem 2.1]{GST2015}}, we have $u\in W_{0}^{1, m}(\Omega)\cap C^{0, \beta}(\overline{\Omega})$ for some $\beta \in (0, 1)$. Using the fact that
\begin{equation*}
\underline{u}(x)= \frac{1}{c}\varphi(x)^{\frac{m-q}{m+p-1}} \quad\mbox{ and } \quad\overline{u}(x)= c\varphi(x)^{\frac{m-q}{m+p-1}},
\end{equation*}
are respectively  sub and supersolutions of \eq{eq8} for some large $c>1$, we easily deduce that 
$$
u\sim \delta(x)^{\frac{m-q}{m+p-1}}.
$$  
Then
$$
\theta(x)= K(x)u^{-p}(x)\sim \delta(x)^{-\frac{mp+(m-1)q}{m+p-1}},
$$
and note that $a= \frac{mp+(m-1)q}{m+p-1}\in \Big(1, 2-\frac{1}{m}\Big)$. By Proposition \ref{p1}(ii) it follows that $u\in W_{0}^{1, \tau}(\Omega)$ for all $\tau \in \Big[m, \frac{m+p-1}{p+q-1}\Big)$.
\end{itemize}

\end{document}